\documentclass[12pt]{article}
\usepackage{amsmath, amsthm, amssymb, amsfonts, graphicx}
\title{The Cross Number of Minimal Zero-sum Sequences in Finite Abelian Groups}
\author{Bumsoo Kim\footnote{bumsook@princeton.edu}\\
Princeton University \\ Department of Mathematics}
\usepackage{tikz}
\usepgflibrary{shapes}

\def\be{\begin{enumerate}}
\def\ii{\item}
\def\ee{\end{enumerate}}
\def\ms{\mathsf}
\def\ord{\operatorname{ord}}
\def\lcm{\operatorname{lcm}}
\def\N{\mathbb{N}}
\def\Z{\mathbb{Z}}

\theoremstyle{plain}
\newtheorem{thm}{Theorem}
\newtheorem{conj}[thm]{Conjecture}
\newtheorem{lem}[thm]{Lemma}

\newtheorem{cor}[thm]{Corollary}

\theoremstyle{definition}
\newtheorem{defn}[thm]{Definition}

\begin{document}
\maketitle

\begin{abstract}
We study the maximal cross number $\mathsf{K}(G)$ of a minimal zero-sum sequence and the maximal cross number $\mathsf{k}(G)$ of a zero-sum free sequence over a finite abelian group $G$, defined by Krause and Zahlten. In the first part of this paper, we extend a previous result by X. He to prove that the value of $\mathsf{k}(G)$ conjectured by Krause and Zahlten hold for $G \bigoplus C_{p^a} \bigoplus C_{p^b}$ when it holds for $G$, provided that $p$ and the exponent of $G$ are related in a specific sense. In the second part, we describe a new method for proving that the conjectured value of $\mathsf{K}(G)$ hold for abelian groups of the form $H_p \bigoplus C_{q^m}$ (where $H_p$ is any finite abelian $p$-group) and $C_p \bigoplus C_q \bigoplus C_r$ for any distinct primes $p,q,r$. We also give a structural result on the minimal zero-sum sequences that achieve this value.
\end{abstract}

\section{Introduction}

The following notations are adapted from \cite{BC}, \cite{H}, \cite{Kriz}.

Let $(G,+)$ be a finite abelian group written additively. For any subset $G_0 \subseteq G$, denote by $\mathcal{G}(G_0)$ to be the multiplicative free abelian group generated by $G_0$. Similarly, we define $\mathcal{F}(G_0) \subseteq \mathcal{G}(G_0)$ to be the multiplicative free abelian monoid generated by $G_0$. A \textit{sequence} over $G_0$ is an element of $\mathcal{F}(G_0)$. We may write elements of $\mathcal{G}(G_0)$ in the form
$$S = \prod_{g \in G_0} g^{v_g(S)}$$
where $v_g: \mathcal{G}(G_0) \rightarrow \mathbb{Z}$ is the \textit{valuation} function for $g$, satisfying $v_g(S)=0$ for all but finitely many $g$ given any fixed $S$, and $v_g(S) \geq 0$ for all $g \in G_0$ if $S \in \mathcal{F}(G_0)$. The identity $1$ of the monoid $\mathcal{F}(G_0)$ is the unique sequence satisfying $v_g(1) = 0$ for all $g \in G_0$. Given two sequences $S,T \in \mathcal{F}(G_0)$, we say that $T$ is a \textit{subsequence} of $S$, or \textit{divides} $S$, if $v_g(T) \leq v_g(S)$ for all $g \in G_0$. In such a case we may also write $T | S$. We say $T$ is a \textit{proper} subsequence, or a $\textit{proper divisor}$ if $T | S$ and $T \neq S$.


The $sum$ function $\sigma : \mathcal{F}(G) \rightarrow G$ is defined on a sequence $S$ as
$$\sigma(S) = \sum_{g \in G} v_g(S) \cdot g.$$


Define the \textit{set of subsums}, or \textit{sumset} $\Sigma(S)$ of a sequence $S$ to be the set
$$\Sigma(S) = \{ {\sigma} (T): ~ 1 \neq T, T | S \}.$$

A sequence $S$ is \textit{zero-sum} if $\sigma(S) = 0$; \textit{zero-sum free} if the only sequence $T | S$ with $\sigma(T) = 0$ is $1$; and \textit{minimal zero-sum}, if it differs from 1, is zero-sum, and has no nontrivial zero-sum sequence as proper divisors.

When $S$ is a sequence over $G$, let $|S| = \sum_{g \in G^{\star}} v_g(S)$. This is referred to as the \textit{length} of $S$. Define the \textit{cross number} of a sequence $S$ to be

$$ \mathsf{k}(S) = \sum_{g \in G} \frac{v_g(S)}{\ord(g)}.$$ 

It is often more natural to study the cross number of a sequence than to study its length, see \cite{BC}.


The \textit{little cross number of} $G$ is defined as
$$ \mathsf{k}(G) = \max \lbrace \mathsf{k}(S): S \text{ is zero-sum free over } G \rbrace.$$

The \textit{cross number of} $G$ is defined as
$$ \mathsf{K}(G) = \max \lbrace \mathsf{k}(S): S \text{ is minimal zero-sum over } G \rbrace.$$

The constants $\mathsf{k}(G)$ and $\mathsf{K}(G)$, and how minimal zero-sum sequences $S$ with $\ms{k}(S) = \ms{K}(G)$ do look like, have been studied in the literature since more than 20 years. 

For a positive integer $n>1$, define $P^{-}(n)$ to be the smallest prime dividing $n$. It can be immediately verified that 
$$\ms{k}(G) + \frac{1}{\exp(G)} \leq \ms{K}(G) \leq \ms{k}(G) + \frac{1}{P^{-}(\exp(G))},$$
by observing that removing any element from a minimal zero-sum sequence leaves a zero-sum free sequence.

Writing $G$ as a direct sum of prime power order cyclic groups
$$ G = \bigoplus_{i=1}^{r} C_{p_i^{\alpha_i}}, $$
we define
$$ \mathsf{k}^\star (G) = \sum_{i=1}^{r} \left( 1 - \frac{1}{p_i^{\alpha_i}}\right),$$
and
$$ \mathsf{K}^\star (G) = \ms{k}^\star (G) + \frac{1}{\exp(G)}.$$ 

These are the conjectured values of $\ms{k}(G)$ and $\ms{K}(G)$ respectively. In this context, facts about $\ms{K}(G)$ imply facts about $\ms{k}(G)$, since if

$$\ms{K}(G) = \ms{K}^\star (G) = \ms{k}^\star (G) + \frac{1}{\exp{G}},$$ we clearly have $\ms{k}(G) = \ms{k}^\star (G)$ as well.

Krause and Zahlten \cite{KZ} conjectured the following. (See \cite{GG}, \cite{Girard} and \cite{H} for the most recent progress on the conjecture.)

\begin{conj}
\label{conj:main}
The equality $\mathsf{K}(G) = \mathsf{K}^{\star}(G)$ holds for all finite abelian groups $G$, and therefore $\mathsf{k}(G) = \mathsf{k}^{\star}(G)$ holds for all finite abelian groups $G$ as well.
\end{conj}

In this paper we study the constants $\mathsf{k}(G)$ and $\ms{K}(G)$, and the structures of the minimal zero-sum sequences $S$ with $\ms{k}(S) = \ms{K}(G)$.

\section{Previous Results}

It is not difficult to see that sequences of sufficient length or cross number over any nontrivial $G$ will not be zero-sum free, so $\mathsf{k}(G)$ and $\mathsf{K}(G)$ are finite for all $G$.

The lower bound $\mathsf{K}(G) \geq \mathsf{K}^{\star}(G)$ is known by construction \cite{KZ}, so to prove Conjecture \ref{conj:main}, it suffices to prove
$$\mathsf{K}(G) \leq \mathsf{K}^{\star} (G).$$
This has already been shown in the following special cases.


\begin{thm}
\label{thm:prevresult}
If $G$ is a group of one of the following forms, then $\mathsf{K}(G) = \mathsf{K}^{\star}(G)$, and thus $\mathsf{k}(G) = \mathsf{k}^{\star}(G)$.
\be
\ii \cite{G2} $G$ is a finite abelian p-group.
\ii \cite{GS2} $G = C_{p^m} \bigoplus C_{p^n} \bigoplus C_q^s$ with distinct primes $p,q$ and $m,n,s \in \mathbb{N}$ 
\ii \cite{GS2} $G = \bigoplus_{i=1}^{r} C_{p_i^{n_i}} \bigoplus C_q^s$ with distinct primes $p_1, p_2, \dots, p_r, q,$ and integers $n_1, \dots, n_r$, $s \in \mathbb{N}$, such that either $r \leq 3$ and $p_1 p_2 \cdots p_r \neq 30$ or $p_k \geq k^3$ for every $1 \leq k \leq r$.
\ee
\end{thm}

He \cite{H} has shown an inductive result on the little cross number $\ms{k}(G)$. In order to state his result, we need the following definition introduced by He.

\begin{defn}
\label{defn:wide}
(See \cite{H}) A prime $p$ is \textit{wide with respect to} an integer $n$ if $p \nmid n$ and the inequality
\begin{equation}
\label{eq:pnwide}
\frac{p}{p-1} \geq \sum_{d|n} \frac{1}{d}
\end{equation}
holds. Writing $n = p_1^{d_1} p_2^{d_2} \cdots p_k^{d_k}$, this inequality can be rewritten as
\begin{equation}
\frac{p}{p-1} \geq \prod_{j=1}^{r} \frac{p_i^{d_i + 1} - 1}{p_i^{d_i +1} - p_i^{d_i}},
\end{equation}
in which case we write $p \prec n$. The empty product is taken to be 1.
\end{defn}

\begin{thm}
\label{thm:heinduct}
(See \cite{H})
Given a finite abelian group $G$, let $\exp(G)$ denote the least common multiple of the orders of all elements of $G$. Given a prime $p$ and a finite abelian group $G$, if $p \prec \exp(G)$
then
\begin{equation}
\label{eq:wideinduct}
\mathsf{k}(C_{p^\alpha} \bigoplus G) = \mathsf{k}(C_{p^\alpha}) + \mathsf{k}(G)
\end{equation}
for all $\alpha \in \mathbb{N}$. In particular, if $\mathsf{k}(G) = \mathsf{k}^\star(G)$ then $\mathsf{k}(C_{p^\alpha} \bigoplus G) = \mathsf{k}^\star(C_{p^\alpha} \bigoplus G)$ as well.
\end{thm}

Combining Theorem~\ref{thm:heinduct} with Theorem~\ref{thm:prevresult}, we can show Conjecture~\ref{conj:main} for more general cases. For $a,b \in \mathbb{N}$, $[a,b]$ denotes the set of integers $\lbrace m: a \leq m \leq b \rbrace$.

\begin{thm}
\label{thm:hegroup}
If $G$ is a finite abelian group of one of the following forms, then $\ms{k}(G) = \ms{k}^{\star}(G)$.
\be
\ii \cite{H} $G = \bigoplus_{i=1}^{r} C_{p_i^{n_i}} \bigoplus {H_q}$ with distinct primes  $p_1, p_2, \dots, p_r, q,$ and integers $n_1, \dots, n_r \in \mathbb{N}$, such that $H_q$ is a $q$-group and if $\exp(H_q) = q^n$, then $p_i \prec p_{i+1}^{n_{i+1}} p_{i+2}^{n_{i+2}} \cdots {p_r}^{n_r} q^n$ for every $1 \leq i \leq r$.
\ii \cite{H} $G = \bigoplus_{i=1}^{r} C_{p_i^{n_i}} \bigoplus C_{p_{r+1}^{n_{r+1}}} \bigoplus C_{p_{r+1}^{n_{r+1}^\star}} \bigoplus {C_q}^s$ with distinct primes \\ $p_1, p_2, \dots, p_r, p_{r+1}, q$ and integers $n_1, \dots, n_r, n_{r+1}, n_{r+1}^\star, s \in \mathbb{N}$, such that $p_i \prec p_{i+1}^{n_{i+1}} \cdots p_{r+1}^{n_{r+1}} q$ for all $i \in [1,r]$ and $n_{r+1} \geq n_{r+1}^\star$.
\ee
\end{thm}

It should be noted that for general forms of groups with more than two prime divisors, Conjecture~\ref{conj:main} has only been proved for groups in which most prime divisors have only ``one generator" (i.e., most Sylow p-subgroups of $G$ are cyclic). Our work focuses mainly on attempting to go beyond this limitation.


\section{Summary of Main Results}

In the first part of this paper, we prove an inductive result building on He's result (\cite{H}). We first introduce the following definition, building on the wideness definition proposed by He \cite{H}.

\begin{defn}
\label{defn:2small} A prime $p$ is \textit{2-small with respect to} $n$ if $p \nmid n$ and
$$\frac{2p+2}{2p+1} > \sum_{d|n} \frac{1}{d},$$
\noindent in which case we write $p \prec^2 n$.
\end{defn}

Our first main theorem bearing on Conjecture~\ref{conj:main} is the following inductive result.


\begin{thm}
\label{thm:result1}
Let $G$ be a finite abelian group and $p$ be a prime satisfying $p \prec^2 \exp(G)$. Then the equality
$$\mathsf{k}(C_{p^m} \bigoplus C_{p^n} \bigoplus G) = \mathsf{k}(C_{p^m}) + \mathsf{k}(C_{p^n}) + \mathsf{k}(G)$$
\noindent holds. In particular, if $\mathsf{k}(G) = \mathsf{k}^{\star}(G)$, then
$$\mathsf{k}(C_{p^m} \bigoplus C_{p^n} \bigoplus G) = \mathsf{k}^{\star} (C_{p^m} \bigoplus C_{p^n} \bigoplus G)$$
as well.
\end{thm}
\noindent In Section~\ref{sec:part1} I give a proof of this theorem.


Our second main theorem pertaining to Conjecture~\ref{conj:main} is the following (non-inductive) result.

\begin{thm}
\label{thm:result2}
Let $G$ be a finite abelian group of one of the following forms.
\be
\ii $G = H_p \bigoplus C_{q^k}$ with distinct primes $p,q$, where $H_p$ is any finite $p$-group.
\ii $G = C_{p} \bigoplus C_{q} \bigoplus {C_r}$ with distinct primes $p,q,r$.
\ee
Then we have
\be
\ii Each minimal zero-sum sequence $U$ over $G$ with $\ms{k}(U) = \ms{K}(G)$ has the form $U = g \prod U_i$ where $U_i$ is a zero-sum free sequence over the $p_i$-primary component of $G$ for each $i \in [1,s]$, where $\{p_1, p_2, \cdots, p_s\}$ is the set of prime divisors of $\exp(G)$.
\ii In particular, Conjecture~\ref{conj:main} holds: $\ms{K}(G) = \ms{K}^\star (G) = \ms{k}^\star (G) + \frac{1}{\exp(G)}$ and each zero-sum free sequence $S$ with $\ms{k}(S) = \ms{k}^\star (G)$ has the form $S = \prod S_i$ where $S_i$ is a zero-sum free sequence over the $p_i$-primary component of $G$ for each $i \in [1,s]$.
\ee
\end{thm}
\noindent The proof of this theorem uses a different apporach, discussed in Section~\ref{sec:part2}.

Both results extend previous results. Past results either had the restriction that (1) $\exp(G)$ has a small number of prime factors or (2) $G$ has only one generator corresponding to each prime except for at most one prime, with certain restrictions on the size of that prime number. Theorem~\ref{thm:result1} extends the proof to cases where $G$ has two generators corresponding to each prime.

Theorem~\ref{thm:result2} applies to all groups of this form with no conditions on the sizes of the primes, and it also solves the associated inverse problem for the case $G = C_p \bigoplus C_q \bigoplus C_r$, where $p,q,r$ are arbitrary distinct primes.


\section{Proof of Theorem 7}
\label{sec:part1}

In this section, we provide an extension of He's work \cite{H}, by showing a result of the form $\mathsf{k}(C_{p^m} \bigoplus C_{p^n} \bigoplus G) = \mathsf{k}(C_{p^m}) + \mathsf{k}(C_{p^n}) + \mathsf{k}(G)$.

We decompose a finite abelian group $G$ in the canonical form
\begin{equation}
\label{eqn:canon}
G = \bigoplus_{i=1}^{r} \bigoplus_{j=1}^{k_i} C_{p_i^{a_{i,j}}},
\end{equation}
where $p_1 < p_2 < \cdots < p_r$ are primes and for each $i \in [1,r]$, we assume $a_{i,1} \geq a_{i,2} \geq \cdots \geq a_{i,k_i}$. If $n > k_i$ then $a_{i,n}$ is taken to be zero.\\
 
The following definition from \cite{KZ} is helpful.
\begin{defn}
\label{defn:dense}
A zero-sum free sequence $S$ over $G$ is \textit{dense} if $\mathsf{k}(S) = \mathsf{k}(G)$ and $|S| = \min \lbrace |T|: \mathsf{k}(T) = \mathsf{k}(G), T \text{ is zero-sum free} \rbrace.$
\end{defn}
Following \cite{BCMP} and \cite{H}, we also make the following definition.

\begin{defn}
\label{defn:amal}
By \textit{amalgamating} a subsequence $T$ of a sequence $S$ we mean replacing $T$ with its sum.
\end{defn}

Amalgamating any subsequence of a zero-sum free sequence keeps it zero-sum free. Thus, for a dense zero-sum free sequence, amalgamation decreases $\mathsf{k}(S)$. Noting this, X. He  demonstrated the following restriction on dense zero-sum free sequences, which becomes a key lemma in his theorem.

\begin{lem}[Amalgamation Lemma; see \cite{H}]
\label{lem:amal}  Let $G$ be a group of the form (\ref{eqn:canon}), and suppose that $a_{i,1} > a_{i,2}$ for some $i \in [1,r]$. Let $l$ be a positive integer divisible by $p_i^{a_{i,2} + 1}$. If $S$ is a dense zero-sum free sequence over $G$, then $S$ contains at most $p_i - 1$ elements of order $l$.
\end{lem}

This lemma, along with another lemma from the same paper, will be used in our proof:


\begin{lem}[Lemma 15 of \cite{H}]
\label{lem:1rep}
Let $G$ be of the form (\ref{eqn:canon}) with $a_{1,1} > a_{1,2}$ and let $a \in [a_{1,2} + 1, a_{1,1}]$.
If $S$ is a dense zero-sum free sequence over $G$ and
$$ p_1 \prec p_2^{a_{2,1}} p_3^{a_{3,1}} \cdots p_r^{a_{r,1}},$$
\noindent then $S$ contains at least $p_1 - 1$ elements of order $p_1^a$.
\end{lem}

Our crucial observation is that by loosening the bound $p_i - 1$, we can obtain results for $a_{i,3}$, and $a_{i,j}$ in general. For this we consider two constants studied in recent literature.


\begin{defn}
\label{defn:psum}
Let $G$ be a finite abelian group with exponent $\exp(G) = e$.
\be
\ii The Erd\H{o}s-Ginzburg-Ziv constant $\ms{s}(G)$ is defined as the smallest integer $l \in \N$ such that every sequence over $G$ of length $|S| \geq l$ has a zero-sum subsequence $T$ of length $|T| = e$.
\ii The invariant $\eta(G)$ is defined as the smallest integer $l \in \N$ such that every sequence $S$ over $G$ of length $|S| \geq l$ has a zero-sum subsequence $T$ of length $|T| \in [1,e]$.
\ee
\end{defn}

The invariants $\ms{s}(G)$ and $\eta(G)$ have been studied since the 1960s, and while the problem is yet to be settled for arbitrary $G$ in general, for rank two groups their precise values are known. Indeed we have (see \cite{GH}, Theorem 5.8.3),


\label{lem:rank2eta}
\begin{lem}
Let $G = C_{n_1} \bigoplus C_{n_2} with 1 \leq n_1 ~ | ~ n_2$. Then
$$ \ms{s}(G) = 2n_1 + 2n_2 - 3 \text{   and   } \eta(G) = 2n_1 + n_2 - 2.$$
\end{lem}

For recent development on these invariants we refer the reader to papers by Gao et. al. (\cite{Gao1, Gao2, Gao3}) and the recent monograph by Grynkiewicz (\cite{Gryn}, Chapter 16).

We recall two more constants that extend $\ms{D}(G)$ and $\eta(G)$ from B. Girard \cite{Girard} which will be important in our proof.
\begin{defn}
\label{defn:Gir}
Given a finite abliean group $G$, denote by $G_d$ the subgroup of $G$ consisting of elements of order dividing $d$.

Denote by $\ms{D}_{(d',d)} (G)$ the smallest integer $t \in \N$ such that every sequence $S$ in $G_d$ with length $|S| \geq t$ contains a nonempty subsequence with sum in $G_{d/d'}$.

Denote by $\eta_{(d',d)}(G)$ the smallest integer $t \in \N$ such that every sequence $S$ in $G_d$ with length $|S| \geq t$ contains a nonempty subsequence $S' | S$ with length $|S'| \leq d'$ and sum in $G_{d/d'}$.
\end{defn}

The theorem I cite from Girard's paper \cite{Girard} is Proposition 3.1:
\begin{thm}[3.1 of \cite{Girard}]
\label{thm:girard}
Let $G = C_{n_1} \bigoplus \cdots \bigoplus C_{n_r}$ with $1 < n_1 | \cdots | n_r \in \N$ be a finite abelian group and $d', d$ be such that $d' | d | \exp(G)$. Then, we have the following two equalities:

\begin{align*}
\ms{D}_{(d',d)} (G) & = \ms{D} (C_{v_1(d',d)} \bigoplus \cdots \bigoplus C_{v_r (d',d)}) \\
\eta_{(d',d)} (G) & = \ms{D} (C_{v_1(d',d)} \bigoplus \cdots \bigoplus C_{v_r (d',d)})
\end{align*}

where $v_i (d',d) = \frac{A_i}{\gcd(A_i, B_i)}, A_i = \gcd(d', n_i), B_i = \frac{\lcm(d,n_i)}{\lcm(d',n_i)}$.
\end{thm}

The relation between these constants and the little cross number is exemplified by the following lemma (``$n$-Amalgamation Lemma").

\begin{lem}
\label{lem:namal} ($n$-Amalgamation lemma) Let $G$ be a finite abelian group expressed in the canonical form (\ref{eqn:canon}), and suppose that $a_{i,n} > a_{i,n+1}$ for some $i \in [1,r]$. Let $l$ be a positive integer divisible by $p_i^{a_{i,n+1} + 1}$. If $S$ is a dense zero-sum free sequence over $G$, then $S$ contains at most $\eta(C_{p_i}^n) - 1$ elements of order $l$.
\end{lem}

\begin{proof} 
This theorem follows from Theorem~\ref{thm:girard}: note that we have $\eta_{(p_i, l)} (G) = \eta (C_{p_i}^n)$ and if $S$ contains $\eta_{(p_i, l)} (G)$ elements of order $l$, there exists a subsequence with length less than or equal to $p_i$ and sum of order dividing $\frac{l}{p_i}$, but this contradicts the density assumption of $S$.
\end{proof}

Applying Lemma~\ref{lem:rank2eta} to $G = C_p^2$ gives $\eta(C_p^2) = 3p-2$. Using this constant, the ``$2$-Amalgamation lemma" can be restated as follows.

\begin{lem}
\label{lem:2amal} ($2$-Amalgamation lemma) Let $G$ be a finite abelian group expressed in the canonical form (\ref{eqn:canon}), and suppose that $a_{i,2} > a_{i,3}$ for some $i \in [1,r]$. Let $l$ be a positive integer divisible by $p_i^{a_{i,3} + 1}$. If $S$ is a dense zero-sum free sequence over $G$, then $S$ contains at most $3p_i-3$ elements of order $l$.
\end{lem}

As a consequence of Lemma~\ref{lem:2amal}, we have a bound on the number of terms in dense sequences. To eliminate them altogether, we need a stronger hypothesis, namely the 2-small condition.


\begin{lem}
\label{lem:2rep} ($2$-Replacement lemma) Let $G$ be of the form (\ref{eqn:canon}) with $a_{1,2} > a_{1,3}$ and let $a \in [a_{1,3} + 1, a_{1,2}]$.
If $S$ is a dense zero-sum free sequence over $G$ and
$$ p_1 \prec^2 p_2^{a_{2,1}} p_3^{a_{3,1}} \cdots p_r^{a_{r,1}},$$
\noindent then $S$ contains at least $2p_1 - 2$ elements of order $p_1^a$.
\end{lem}
\begin{proof}

Assume the statement does not hold, and choose the largest $a$ in $[a_{1,3} + 1, a_{1,2}]$ such that there are at most $2p_1 - 3$ elements of order $p_1^a$ in $S$.

For each positive integer $d$ satisfying $p_1^a | d$ and $d | \exp(G)$, write
$$S_d = \prod_{\ord(g) = d} g^{v_g(S)}, $$
the sequence of elements with order $d$.

By Lemma~\ref{lem:2amal}, $|S_d| \leq 3p_1 - 3$. Now call an integer $d$ \textit{full} if the following conditions hold:

\be
\ii $p_1^a || d$. (That is, $p_1^a | d$ but $p_1^{a+1} \nmid d$.)
\ii $2p_1-1 \leq |S_d| \leq 3p_1-3$.
\ee

Since $|S_d| \geq 2p_1-1$ for each full $d$, we can choose a subsequence $T_d$ of $S_d$ such that $\sigma(T_d)$ has order dividing $\frac{d}{p_1}$, again by Theorem~\ref{thm:girard}, because $D_{(p_1, d)}(G) = 2p_1-1$.

Given the original sequence $S$, denote the full integers corresponding to $S$ as $d_1 , d_2 , \ldots, d_n$, and let

\be
\ii $D_1$ denote the set of full integers,
\ii $D_2$ denote the set of integers $d$ satisfying $p_1^a || d | \exp(G)$ but is not full, excluding $p_1^a$.
\ii $D_3$ denote the set of integers $d$ satisfying $p_1^{a+1} | d | \exp(G)$, but is not divisible by $p_1^{a_{1,2} + 1}$
\ii $D_4$ denote the set of integers $d$ satisfying $p_1^{a_{1,2} + 1} | d | \exp(G)$.
\ee

We distinguish $D_3$ from $D_4$ because there are at most $3p_1 - 3$ elements of order $d \in D_3$ in $S$ but there are at most $p_1 - 1$ elements of order $d \in D_4$ in $S$. This distinction will become important later.

From the definitions of $D_i$ it is clear that
\begin{align*}
\{ d: p_1^a | d | \exp(G) \} & = \{p_1^a \} \cup D_1 \cup D_2 \cup D_3 \cup D_4 \\
\{ d: p_1^a | d | \exp(G), p_1^{a_{1,2} + 1} \nmid d\} & = \{p_1^a \} \cup D_1 \cup D_2 \cup D_3 \\
\{d: p_1^a || d | \exp(G) \} & = \{p_1^a\} \cup D_1 \cup D_2.
\end{align*}

Consider $S'$ the subsequence consisting of the elements with order divisible by $p_1^a$. Remove the subsequence $S'$ from $S$, and replace it by the sequence $S_1 S_2$, where $S_1$ is a sequence of $2p_1 - 2$ elements of each order $p_1^a, p_1^{a+1}, \dots, p_1^{a_{1,2}}$ and $p_1 - 1$ elements of each order $p_1^{a_{1,2}+1}, \dots, p_1^{a_{1,1}}$, and $S_2$ is $\sigma(T_{d_1}) \sigma(T_{d_2}) \cdots \sigma(T_{d_n})$. Choose $S_1$ as follows: if $e_1, e_2$ are respectively generators of the component $C_{p_1^{a_{1,1}}}, C_{p_1^{a_{1,2}}}$, then define

$$S_1 = \prod_{k=0}^{a_{1,1} - a} [p_1^{k}e_1]^{p_1 -1} \prod_{l=0}^{a_{1,2}-a}[p_1^{l} e_2]^{p_1 -1}.$$

\noindent The result of the replacement is the sequence $T = S(S')^{-1} S_1 S_2$.

Since no subsequence sum of $S(S')^{-1} S_2$ has order divisible by $p_1^a$, but all subsequence sums of $S_1$ have order divisible by $p_1^a$, if $T$ contains a zero-sum subsequence, the subsequence cannot contain elements from $S_1$. But then, $S(S')^{-1}S_2$ is a zero-sum free sequence because it is the result of removing and amalgamating some terms of $S$. Thus $T$ is a zero-sum free sequence.

If we can show that $\mathsf{k}(T) > \mathsf{k}(S)$, or equivalently $\mathsf{k}(S_1) + \mathsf{k}(S_2) > \mathsf{k}(S')$, we will have a contradiction because $T$ is a zero-sum free sequence with larger cross number than $S$. Bounding $\mathsf{k}(S')$ by the number of terms of each order using Lemma~\ref{lem:amal}, Lemma~\ref{lem:2amal}, and the fullness criterion, we have

\begin{align*}
\mathsf{k}(S_1) & = \sum_{t=a}^{a_{1,2}} \frac{2p_1-2}{p_1^t} + \sum_{t={a_{1,2}+1}}^{a_{1,1}} \frac{p_1-1}{p_1^t} \\
\mathsf{k}(S_2) & \geq \sum_{d \in D_1} \frac{1}{d/p_1} = \sum_{d \in D_1} \frac{p_1}{d} \\
\mathsf{k}(S') & \leq \left( \frac{2p_1-3}{p_1^a} + \sum_{d \in D_1} \frac{3p_1-3}{d} + \sum_{d \in D_2} \frac{2p_1-2}{d} + \sum_{d \in D_3} \frac{3p_1-3}{d} + \sum_{d \in D_4} \frac{p_1 - 1}{d} \right).
\end{align*}

Let $m = \frac{\exp{G}}{p_1^{a_{1,1}}}$ and $X = \sum_{d|m} \frac{1}{d}$. After reorganizing terms, we have that

\begin{align*}
& \mathsf{k}(S_1) + \mathsf{k}(S_2) - \mathsf{k}(S') \\
& \geq \ms{k}(S_1) - \frac{2p_1 - 3}{p_1^a} - \sum_{d \in D_1} \frac{2p_1 - 3}{d} - \sum_{d \in D_2} \frac{2p_1 - 2}{d} - \sum_{d \in D_3} \frac{3p_1-3}{d} - \sum_{d \in D_4} \frac{p_1 - 1}{d} \\
& \geq \frac{1}{p_1^a} + \sum_{t=a + 1}^{a_{1,2}} \frac{2p_1-2}{p_1^t} + \sum_{t={a_{1,2}+1}}^{a_{1,1}} \frac{p_1-1}{p_1^t} - \sum_{d \in D_1 \cup D_2} \frac{2p_1 - 2}{d} - \sum_{d \in D_3} \frac{3p_1 - 3}{d} - \sum_{d \in D_4} \frac{p_1 - 1}{d}  \\
& \geq \frac{1}{p_1^a} - \frac{2p_1-2}{p_1^a}(X - 1) - (3p_1-3)(\frac{1}{p_1^{a+1}} + \frac{1}{p_1^{a+2}} + \cdots + \frac{1}{p_1^{a_{1,1}}})(X - 1)
\end{align*}

To prove that the last term is not less than $0$, it suffices to show (after multiplying both sides by $p_1^a$ and replacing $1/p_1 + 1/p_1^2 + \cdots + 1/p_1^{a_{1,1}-a}$ with $\frac{1}{p_1-1}$, which is larger), we have

$$\frac{2p_1+2}{2p_1+1} \geq X,$$

\noindent which is exactly the 2-smallness assumption.

\end{proof}

Lemma~\ref{lem:2rep} makes the inequality in the definition of 2-smallness assumption necessary. 

Combining the above results, we can prove Theorem~\ref{thm:result1}.

\begin{proof}[Proof of Theorem~\ref{thm:result1}]

The proof is similar to the proof of Theorem 7 of \cite{H}.

Let $G$ be a finite abelian froup and $p$ be a prime satisfying $p <^2 \exp (G)$. Let also $S$ be a dense zero-sum free sequence over $G$. Let $H = C_{p^m} \bigoplus C_{p^n} \bigoplus G$ and assume without loss of generality that $m \geq n$. 

By Lemma~\ref{lem:2rep}, there are $2p-2$ elements of order $p, p^2, \dots, p^n$, and by Lemma~\ref{lem:1rep}, there are $p-1$ elements of order $p^{n+1}, \dots, p^m$.

Let $G' = C_{p^m} \bigoplus C_{p^n}$ be the $p$-component of $H$, and let $S'$ be the subsequence of $S$ consisting of elements of $G'$. Then since elements of $S$ of order order $p^i$ are all in $S'$, we have $$\mathsf{k}(S') \geq \sum_{i=1}^{n} \frac{2p-2}{p^i} + \sum_{j=n+1}^{m} \frac{p-1}{p^j} = \mathsf{k}(G')$$ (we know the value of $\mathsf{k}(G')$ from Theorem~\ref{thm:prevresult}), but by definition, $\mathsf{k}(G')$ is the maximal cross number among zero-sum free sequences over $G'$, and $S'$ is a zero-sum free sequence over $G'$, so equality must hold.

Since the cross number is maximal, the sumset $\Sigma(S')$ of $S'$ must contain all nonzero elements of $G'$, for otherwise we can add elements of $G'$ to $S'$ and still have a sequence with larger cross number (violating $\ms{k}(S') = \ms{k}(G')$).

But then $S(S')^{-1}$ cannot have any subsums lying in $G'$, for otherwise we could form a zero-sum subsequence of $S$ with a subsequence in $S'$.
Therefore, even after projecting $H \mapsto G$, the image of $S(S')^{-1}$ is still zero-sum free. As a result, $\mathsf{k}(S(S')^{-1}) \leq \mathsf{k}(G)$ since projection cannot decrease the cross number. Then we have
$$\mathsf{k}(S) \leq \mathsf{k}(G') + \mathsf{k}(G) = \mathsf{k}(C_{p^m} \bigoplus C_{p^n}) + \mathsf{k}(G) = \mathsf{k}(C_{p^m}) + \mathsf{k}(C_{p^n}) + \mathsf{k}(G),$$

\noindent and the proof is complete as desired by observing that $\mathsf{k}(G \bigoplus H) \geq \mathsf{k}(G) + \mathsf{k}(H)$ for any two groups $G,H$. (The merger of the two maximal sequences over $G$ and $H$ is still zero-sum free in $G \bigoplus H$)
\end{proof}


\section{Proof of Theorem~\ref{thm:result2}}
\label{sec:part2}

In this section, we provide an alternative approach by merging the terms of a sequence to simplify the problem. The main idea of this approach is that given a group $G$, we attempt to concatenate some terms of given order with some terms of another order, so that we can ``increase" the number of terms in the zero-sum free sequence with a given prime (power) order and construct an inequality. In this section, we prove Theorem~\ref{thm:result2}.

To prove Theorem~\ref{thm:result2}, we proceed as follows: we first prove the \textit{little cross number} conjecture of Conjecture~\ref{conj:main}, that is, $\ms{k}(G) = \ms{k}^\star(G)$ for the groups listed. Looking at the equality conditions, we verify the inverse problem, that is, the structure of zero-sum free sequences $S$ for which $\ms{k}(S) = \ms{k}(G)$. Then we adjust the proof to show that Conjecture~\ref{conj:main} holds for minimal zero-sum sequences $U$ too.

\begin{proof}[Proof of Theorem~\ref{thm:result2}, Part 1]

We show Theorem~\ref{thm:result2} for $G = H_p \bigoplus C_{q^m}$ where $p,q$ are distinct primes and $H_p$ is a $p$-group. We proceed as follows.

\textit{Step 1.} We first wish to show that $\mathsf{K}(G) = \mathsf{K}^{\star}(G)$ for $G = H_p \bigoplus C_{q^m}$ where $p,q$ are distinct primes and $H_p$ is a $p$-group. Denote $\exp(H_p) = p^k$.

Fix a zero-sum free sequence $S$ in $G$. It suffices to show $\mathsf{k}(S) \leq \mathsf{k}^{\star}(G).$ Denote by $a(n)$ the number of elements of order $n$ in the sequence $S$. Write $S = S_p T_0 T_1 T_2 \cdots T_k$, where $S_p$ is the subsequence consisting of all elements of order $p^i$ for some $i$, and $T_i$ is the subsequence consisting of all elements of order $p^i q^j$ for some $j \geq 1$.

Projecting the subsequence $T_0 T_1$ onto the $C_{q^m}$-coordinate (denote this projection $\tau: G \mapsto C_{q^m}$), the resulting sequence over $C_{q^m}$ has cross number $\mathsf{k}(\tau(T_0)) + \mathsf{k}(\tau(T_1))$. But since $\mathsf{k}(C_{q^m}) = 1 - \frac{1}{q^m}$, every subsequence with cross number at least $1$ in $C_{q^m}$ will have a zero-sum subsequence with cross number at most $1$. Therefore, we can find $\lfloor \mathsf{k}(\tau(T_0)) + \mathsf{k}(\tau(T_1)) \rfloor$ nonoverlapping zero-sum subsequences in the projection of $T_0 T_1$.

Replace the preimage of these zero-sum subsequences with its respective sum: denote this replacement $T_0 T_1 \mapsto Q_1 R_1$, where $Q_1$ is the ``replaced" sums, and $R_1$ is the ``leftover" elements that have not been replaced. Since $S$ is a zero-sum free sequence, elements of $Q_1$ have order $p$, and $Q_1$ is a sequence of length $\lfloor \mathsf{k}(\tau(T_0)) + \mathsf{k}(\tau(T_1)) \rfloor$. The replacement is expressed as follows:

$$ S = S_p T_0 T_1 T_2 \cdots T_k \mapsto S_p Q_1 R_1 T_2 T_3 \cdots T_k .$$

Now we inductively repeat this process. For each $i < k$, assume that $S$ has been replaced by $S_p Q_1 Q_2 \cdots Q_i R_i T_{i+1} \cdots T_k$. Project the subsequence $R_i T_{i+1}$ onto the $C_{q^m}$-coordinate (denote this projection $\tau: G \mapsto C_{q^m}$). The resulting sequence over $C_{q^m}$ has cross number $\mathsf{k}(\tau(R_i)) + \mathsf{k}(\tau(T_{i+1}))$. Arguing as above, we can find $\lfloor \mathsf{k}(\tau(R_i)) + \mathsf{k}(\tau(T_{i+1})) \rfloor $ nonoverlapping zero-sum subsequences in the projection of $R_i T_{i+1}$.

Replace the preimage of these zero-sum subsequences with its respective sum: denote this replacement $R_i T_{i+1} \mapsto Q_{i+1} R_{i+1}$, where $Q_{i+1}$ is the ``replaced" sums, and $R_{i+1}$ is the ``leftover" elements that have not been replaced as a part of a sum. Elements of $Q_{i+1}$ have order dividing $p^{i+1}$ since their projection onto the $C_{q^m}$ coordinate is zero.

Repeating this process for $1 \leq i < k$, we can apply the following transformation to $S$:

$$S = S_p T_0 T_1 T_2 \cdots T_k \mapsto S_p Q_1 Q_2 \cdots Q_k R_k = S'.$$

\noindent Note that the only transformation on $S$ was replacing some groups of elements by their sum. So the zero-sum free property is preserved on $S'$, but now all elements of the subsequence $S_p Q_1 \cdots Q_k$ are in $H_p$ (they have zero $C_{q^m}$-component).

Thus we have the inequality

\begin{equation}
\label{eqn:replace}
\mathsf{k}(H_p) \geq \mathsf{k}(S_p Q_1 \cdots Q_k) = \mathsf{k}(S_p) + \mathsf{k}(Q_1) + \cdots + \mathsf{k}(Q_k).
\end{equation}

\noindent We know that $\mathsf{k}(S_p) = \sum_{i=1}^{k} \frac{a(p^i)}{p^i}$, so it remains to determine $\mathsf{k}(Q_i)$ for each $i$. But since each $Q_i$ consists of elements of order dividing $p^i$, we have $$\mathsf{k}(Q_i) \geq \frac{|Q_i|}{p^i}.$$ Thus it suffices to determine $|Q_i|$. But from the construction process and from the definition, we know the following. 
\begin{align*}
|Q_1| = & \lfloor \mathsf{k}(\tau(T_0)) + \mathsf{k}(\tau(T_1)) \rfloor \\
\mathsf{k}(\tau(R_1)) = & \{ \mathsf{k}(\tau(T_0)) + \mathsf{k}(\tau(T_1)) \} \\
|Q_i| = & \lfloor \mathsf{k}(\tau(R_{i-1})) + \mathsf{k}(\tau(T_i)) \rfloor ~~~ (i \geq 2) \\
\mathsf{k}(\tau(R_i)) = & \{ \mathsf{k}(\tau(R_{i-1})) + \mathsf{k}(\tau(T_i)) \} ~~~ (i \geq 2) \\
\mathsf{k}(\tau(T_i)) = & \sum_{j=1}^{m} \frac{a(p^i q^j)}{q^j},
\end{align*}
where $\{ x\}$ denotes the fractional part of $x$ $(\{ x\}= x - \lfloor x \rfloor)$.

For convenience we write $R_0 = T_0$: this allows us to use the $Q_i$ equations even when $i=1$. Rewrite Inequality~\ref{eqn:replace} as
\begin{equation}
\label{eqn:pqineq}
\mathsf{k}(H_p) \geq \mathsf{k}(S_p) + \mathsf{k}(Q_1) + \cdots + \mathsf{k}(Q_k) \geq \sum_{i=1}^{k} \frac{a(p^i)}{p^i} + \sum_{i=1}^k \frac{|Q_i|}{p^i}.
\end{equation}
To simplify the sum $\sum_{i=1}^k \frac{|Q_i|}{p^i}$, we use the following subclaim.

\begin{lem}
\label{lem:floorsum1}
Given a sequence $t_1, t_2, \ldots, t_n \in \mathbb{Q}$, let $s_1, s_2, \ldots, s_n$ be a sequence satisfying
\begin{align*}
s_1 & = t_1 \\
s_{i} & = \{ s_{i-1} \} + t_{i} ~~ (i \geq 2).
\end{align*}
\noindent If we further assume that $s_i b, t_i b \in \Z$ for all $i$ for some $b \in \Z_{>0}$, then the inequality

$$ \sum_{i=1}^{n} \frac{\lfloor s_i \rfloor}{p^i} \geq \sum_{i=1}^{n} \frac{t_i}{p^i} + \frac{1}{p}\left(\frac{1}{b} - 1\right) $$
\noindent holds, with equality iff $s_i + \frac{1}{b} \in \Z$ for all $i$.
\end{lem}

\begin{proof}[Proof of Lemma]
We first note that if $xb  \in \mathbb{Z}$, then the inequality
\begin{equation}
\label{eqn:floor}
\lfloor x \rfloor \geq x + \frac{1}{b} - 1
\end{equation}
holds, with equality if and only if $x + \frac{1}{b} \in \mathbb{Z}$.

Proceed by induction on $n$. When $n=1$, the inequality is just $\frac{\lfloor t_1 \rfloor}{p} \geq \frac{t_1 + 1/b - 1}{p}$ which holds by Inequality~\ref{eqn:floor}. Now assume the claim holds for $n=k-1$: for $n=k$, we transform the sum as follows:
\begin{align*}
\sum_{i=1}^k \frac{\lfloor s_i \rfloor}{p^i} & = \frac{\lfloor t_1 \rfloor}{p} + \sum_{i=2}^{k} \frac{\lfloor s_i \rfloor}{p^i} \\
& = \frac{\lfloor t_1 \rfloor(p-1)}{p^2} + \frac{\lfloor t_1 \rfloor}{p^2} + \frac{1}{p} \sum_{i=1}^{k-1} \frac{\lfloor s_{i+1} \rfloor}{p^i} \\
& \geq \frac{(t_1 + 1/b - 1)(p-1)}{p^2} + \frac{\lfloor t_1 \rfloor}{p^2} + \frac{1}{p}\left( \frac{\{s_1\}}{p} +  \sum_{i=1}^{k-1} \frac{t_{i+1}}{p^{i}} + \left(\frac{1}{b} -1\right) \frac{1}{p}\right) \\
& = \frac{t_1 (p-1)}{p^2} + \frac{\lfloor t_1 \rfloor}{p^2} + \frac{\{t_1\}}{p^2} + \sum_{i=2}^{k} \frac{t_i}{p^i} + \frac{1}{p}\left(\frac{1}{b} - 1\right) \\
& = \sum_{i=1}^{k} \frac{t_i}{p^i} + \frac{1}{p}\left(\frac{1}{b} - 1\right).
\end{align*}
\noindent Here in the inequality, we transform the term $\frac{\lfloor t_1 \rfloor(p-1)}{p^2}$ using Inequality~\ref{eqn:floor}, and apply the inductive hypothesis to the sequence $s_2, s_3, \ldots, s_n$.

Thus Lemma~\ref{lem:floorsum1} holds for $n=k$ as well.
\end{proof}
Now define the sequence $\{s_i\}$ as $s_1 = \mathsf{k}(\tau(T_0)) + \mathsf{k}(\tau(T_1))$ and $s_{i+1} = \{s_i\} + \mathsf{k}(\tau(T_i))$. Then we have $\lfloor s_i \rfloor = |Q_i|$ and $\{ s_i \} = \mathsf{k}(\tau(R_i))$. Noting that $q^m s_i \in \mathbb{Z}$ for all $i$, we apply Lemma~\ref{lem:floorsum1} to this sequence to obtain
$$ \sum_{i=1}^{k} \frac{|Q_i|}{p^i} \geq  \frac{\mathsf{k}(\tau(T_0)) + \mathsf{k}(\tau(T_1))}{p} + \sum_{i=2}^{k} \frac{\mathsf{k}(\tau(T_i))}{p^i} + \left(\frac{1}{q^m} - 1\right) \frac{1}{p}.$$
Using this inequality, Inequality~\ref{eqn:pqineq} becomes
\begin{align*}
\mathsf{k}(H_p) & \geq \mathsf{k}(S_p) + \mathsf{k}(Q_1) + \cdots + \mathsf{k}(Q_k) \geq \sum_{i=1}^{k} \frac{a(p^i)}{p^i} + \sum_{i=1}^k \frac{|Q_i|}{p^i} \\
& \geq \sum_{i=1}^{k} \frac{a(p^i)}{p^i} + \frac{\mathsf{k}(\tau(T_0))}{p} + \sum_{i=1}^{k} \frac{\mathsf{k}(\tau(T_i))}{p^i} + \frac{1}{p}\left(\frac{1}{q^m} - 1\right) \\
& = \sum_{i=1}^{k} \frac{a(p^i)}{p^i} + \sum_{j=1}^{m} \frac{a(q^j)}{pq^j} + \sum_{i=1}^{k} \sum_{j=1}^{m} \frac{a(p^i q^j)}{p^i q^j} + \frac{1}{p}\left(\frac{1}{q^m} - 1\right).
\end{align*}
\noindent Note that the right-hand side of the above inequality resembles the cross number of this sequence. In fact, if we add the inequality

\begin{equation}
\label{eqn:8-1IH}
\mathsf{k}(C_{q^m}) \frac{p-1}{p} = \left(1 - \frac{1}{q^m}\right)\left(1 - \frac{1}{p}\right) \geq \sum_{j=1}^{m} \frac{a(q^j)(p-1)}{pq^j}
\end{equation}
\noindent to both sides, the right hand side becomes $\sum_{d|p^k q^m} \frac{a(d)}{d} = \mathsf{k}(S)$.

Thus we have
$$\mathsf{k}(S) \leq \mathsf{k}(H_p) + \left(1 - \frac{1}{q^m}\right)\left(1 - \frac{1}{p}\right) - \left(\frac{1}{q^m} - 1\right) \frac{1}{p} = \mathsf{k}(H_p) + 1 - \frac{1}{q^m} = \mathsf{k}^{\star}(G),$$

\noindent which implies that $\ms{k}(G) = \ms{k}^\star(G)$.

\textit{Step 2.} We now show that each zero-sum free sequence $S$ with $\ms{k}(S) = \ms{k}^\star (G)$ has the form $S = S_p S_q$ where $S_p$ is a zero-sum free sequence over the $p$-primary component of $G$ and $S_q$ is a zero-sum free sequence over the $q$-primary component of $G$.

Write $G = H_p \bigoplus C_{q^m}$. To show this, first denote by $S_q$ the subsequence of all elements of $S$ of order $q^j$ for some $j$. Since we added inequalities to prove $\ms{k}^\star(G) \leq \ms{k}(S) \leq \ms{k}^\star(G)$, equality must hold everywhere. Then equality must hold in Equation~\ref{eqn:8-1IH}, which implies that $\sum \frac{a(q^j)}{q^j} = \ms{k}(C_{q^m})$. But the left-hand-side is $\ms{k}(S_q)$. Clearly, $S$ zero-sum free implies $S_q$ is zero-sum free, so $\ms{k}(S_q) = \ms{k}(C_{q^m})$ implies the sumset $\Sigma(S_q)$ contains all nonzero elements of $C_{q^m}$.

But then $S(S_q)^{-1}$ cannot have any subsums lying in $C_{q^m}$, for otherwise we could form a zero-sum subsequence in $S$ with a subsequence in $S_q$. Therefore, even after projecting $S(S_q)^{-1}$ to the $H_p$-component, the sequence is still zero-sum free. Denote the projected sequence as $R$, then $\ms{k}(H_p) \geq \ms{k}(R) \geq \ms{k}(S(S_q)^{-1})$ with equality iff all entries of $S(S_q)^{-1}$ have zero $C_{q^m}$-coordinate - that is, all entries have order $p^i$ for some $i$.

Then we have $\ms{k}(H_p) \geq \ms{k}(S(S_q)^{-1}) = \ms{k}(S) - \ms{k}(S_q) = \ms{k}(H_p) + \ms{k}(C_{q^m}) - \ms{k}(C_{q^m}) = \ms{k}(H_p)$. Since the leftmost side and the rightmost side are equal, equality must hold everywhere and $S = (S(S_q^{-1}))(S_q)$ can be decomposed into a $p$-primary component and a $q$-primary component.


\textit{Step 3.} Now we proceed to prove that $\ms{K}(G) = \ms{K}^\star(G)$. Choose a minimal zero-sum sequence $U$ with $\ms{k}(U) = \ms{K}(G)$.

Proceed with the projection argument as we did in Step 1 to transform $U$ with

$$U = U_p T_0 T_1 \cdots T_k \mapsto U_p Q_1 Q_2 \cdots Q_k R_k = U'.$$

Since the only transformation on $U$ was replacing some groups of elements by their sum, $U'$ should also be minimal zero-sum. if $|R_k| > 0$, then $U_p Q_1 Q_2 \cdots Q_k$ is a zero-sum free sequence over $H_p$, whence we have
$$\ms{k}(H_p) \geq \ms{k}(U_p) + \ms{k}(Q_1) + \cdots + \ms{k}(Q_k),$$
which upon transformation gives $\ms{K}^{\star}(G) \geq \ms{k}^{\star}(G) \geq \ms{k}(S)$.

Now if $|R_k| = 0$, this is equivalent to the statement that $|Q_k| = \ms{k}(\tau(R_{k-1})) + \ms{k}(\tau(T_k))$ (we can remove the floor of the last number in the sum.) Then we can modify Lemma~\ref{lem:floorsum1} as follows:

\begin{lem}
\label{lem:floorsum2}
Given a sequence $t_1, t_2, \ldots, t_n \in \mathbb{Q}$, let $s_1, s_2, \ldots, s_n$ be a sequence satisfying
\begin{align*}
s_1 & = t_1 \\
s_{i} & = \{ s_{i-1} \} + t_{i} ~~ (i \geq 2).
\end{align*}
\noindent If we further assume that $s_i b, t_i b \in \mathbb{Z}$ for all $i$ for some $b \in \mathbb{Z}_{>0}$ and $s_n \in \mathbb{Z}$, then the inequality

$$ \sum_{i=1}^{n} \frac{\lfloor s_i \rfloor}{p^i} \geq \sum_{i=1}^{n} \frac{t_i}{p^i} + (\frac{1}{p} - \frac{1}{p^n})\left(\frac{1}{b} - 1\right) $$
\noindent holds, with equality iff $s_i + \frac{1}{b} \in \Z$ for all $i < n$.
\end{lem}

\begin{proof}[Proof of Lemma]
We proceed by induction. When $n=1$, $s_1 \in \mathbb{Z}$ so both sides are clearly equal. Now assume the lemma holds for $n=k-1$. For $n=k$, we do the same transformation as we did in the proof of Lemma~\ref{lem:floorsum1}:
\begin{align*}
\sum_{i=1}^k \frac{\lfloor s_i \rfloor}{p^i} & = \frac{\lfloor t_1 \rfloor}{p} + \sum_{i=2}^{k} \frac{\lfloor s_i \rfloor}{p^i} \\
& \geq \frac{(t_1 + 1/b - 1)(p-1)}{p^2} + \frac{\lfloor t_1 \rfloor}{p^2} \\
& + \frac{1}{p}\left( \frac{\{s_1\}}{p} +  \sum_{i=1}^{k-1} \frac{t_{i+1}}{p^{i}} + \left(\frac{1}{b} -1\right) \left( \frac{1}{p} - \frac{1}{p^{k-1}}\right)\right) \\
& = \frac{t_1 (p-1)}{p^2} + \frac{\lfloor t_1 \rfloor}{p^2} + \frac{\{t_1\}}{p^2} + \sum_{i=2}^{k} \frac{t_i}{p^i} + \left(\frac{1}{p} - \frac{1}{p^k}\right)\left(\frac{1}{b} - 1\right) \\
& = \sum_{i=1}^{k} \frac{t_i}{p^i} + \left(\frac{1}{p} - \frac{1}{p^k}\right)\left(\frac{1}{b} - 1\right).
\end{align*}
which completes the proof.
\end{proof}

Using this lemma and the fact that $U_p Q_1 Q_2 \cdots Q_k$ is a minimal zero-sum sequence over $H_p$, we have
\begin{align*}
\ms{K}(H_p) & \geq \ms{k}(U_p) + \mathsf{k}(Q_1) + \cdots + \mathsf{k}(Q_k) \geq \sum_{i=1}^{k} \frac{a(p^i)}{p^i} + \sum_{i=1}^k \frac{|Q_i|}{p^i} \\
& \geq \sum_{i=1}^{k} \frac{a(p^i)}{p^i} + \frac{\mathsf{k}(\tau(T_0))}{p} + \sum_{i=1}^{k} \frac{\mathsf{k}(\tau(T_i))}{p^i} + \left(\frac{1}{p} - \frac{1}{p^k} \right)\left(\frac{1}{q^m} - 1\right) \\
& = \sum_{i=1}^{k} \frac{a(p^i)}{p^i} + \sum_{j=1}^{m} \frac{a(q^j)}{pq^j} + \sum_{i=1}^{k} \sum_{j=1}^{m} \frac{a(p^i q^j)}{p^i q^j} + \left(\frac{1}{p} - \frac{1}{p^k} \right)\left(\frac{1}{q^m} - 1\right).
\end{align*}

Now, the sequence consisting of elements of order $q^j$ $(1 \leq j \leq m)$ cannot be the entire sequence $U$ (since $U$ is maximal, there must be terms of order divisible by $p$). Thus such elements form a zero-sum free sequence, and thus we have the inequality
\begin{equation}
\label{eqn:8-1IH2}
\mathsf{k}(C_{q^m}) \frac{p-1}{p} = \left(1 - \frac{1}{q^m}\right)\left(1 - \frac{1}{p}\right) \geq \sum_{j=1}^{m} \frac{a(q^j)(p-1)}{pq^j}
\end{equation}

Add this inequality to the inequality above to have
\begin{align*}
\mathsf{k}(U) & \leq \mathsf{K}(H_p) + \left(1 - \frac{1}{q^m}\right)\left(1 - \frac{1}{p}\right) - \left(\frac{1}{q^m} - 1\right) \left(\frac{1}{p} - \frac{1}{p^k}\right) \\
& = \mathsf{k}(H_p) + 1 - \frac{1}{q^m} + \frac{1}{p^k q^m} = \mathsf{K}^{\star}(G)
\end{align*}
which shows that $\ms{K}(G) = \ms{K}^\star(G)$, as desired.


\textit{Step 4.} Now we prove the structural result: that is, each minimal zero-sum sequence $U$ over $G$ with $\ms{k}(U) = \ms{K}(G)$ has the form $U = g U_p U_q$ where $U_p$ is a zero-sum free sequence of the $p$-primary component of $G$, and analogously for $U_q$.

Recall that $G = H_p \bigoplus C_{q^m}$. Assume we have a minimal zero-sum sequence $U$ over $G$ with $\ms{k}(U) = \ms{K}(G) = \ms{k}(H_p) + \ms{k}({C_q}^m) + \frac{1}{p^k q^m}$ where $p^k = \exp(H_p)$. Write $U = U_p U_q R$, where $U_p$ is the subsequence of $U$ with elements of order $p^i$ for some $i$, $U_q$ is the subsequence of $U$ with elements of order $q^j$ for some $j$, and $R$ is everything else. Since we added inequalities to prove $\ms{k}(S) \leq \ms{K}^\star(G)$, equality must hold everywhere. Then equality must hold in Equation~\ref{eqn:8-1IH2}, which implies that $\sum \frac{a(q^j)}{q^j} = \ms{k}(C_{q^m})$. But the left-hand-side is $\ms{k}(U_q)$. Clearly, $U$ minimal zero-sum implies $U_q$ is zero-sum free, so $\ms{k}(U_q) = \ms{k}(C_{q^m})$ implies the sumset $\Sigma(U_q)$ contains all nonzero elements of $C_{q^m}$.

But then $U_p R$ cannot have any proper subsequence whose sum lies in $C_{q^m}$, for otherwise we could form a zero-sum proper subsequence of $S$ by adding it with a subsequence in $U_q$. Therefore, after projecting $U_p R$ to the $H_p$-component, the sequence is still minimal zero-sum. (It should be zero-sum because $U$ was initially zero-sum and $U_q$ has zero $H_p$-component.) If we denote the projection of $R$ as $R_p$, then we have
\begin{align*}
\ms{K}(H_p) & = \ms{k}(H_p) + \frac{1}{p^k} \geq \ms{k}(S_p R_p) \\
& = \ms{k}(S_p) + \ms{k} (R_p) \\
& = \ms{k}(S) - \ms{k}(S_q) - \ms{k}(R) + \ms{k}(R_p) \\
& = \ms{k}(H_p) + \frac{1}{p^k q^m} - \ms{k}(R) + \ms{k}(R_p) \\
\end{align*}
and thus $\ms{k}(R_p) - \ms{k}(R) \leq \frac{1}{p^k} - \frac{1}{p^k q^m}$.

Now write $R = r_1 r_2 \cdots r_t$, and $\ord(r_i) = p^{a_i} q^{b_i}$, with $1 \leq a_i \leq k$ and $1 \leq b_i \leq m$. Then the above inequality becomes

$$\frac{1}{p^k} - \frac{1}{p^k q^m} \geq \sum_{i=1}^{t} \left( \frac{1}{p^{a_i}} - \frac{1}{p^{a_i} q^{b_i}} \right).$$

Multiply $p^k$ to both sides to have $1 - \frac{1}{q^m} \geq \sum_{i=1}^{t} p^{k - a_i} ( 1 - \frac{1}{q^{b_i}})$. If $t \geq 2$ then the right-hand side is not less than 1, so the inequality is violated: thus $t=1$, and we must also have $a_1 = k$ in that case.

Thus $R$ is actually a one-element sequence with its $H_p$-projection order $p^k$. Since we have that the projection of $U_q R$ onto $C_{q^m}$ is zero-sum and $U_q$ has cross number $\ms{k}(U_q) = \ms{k}(C_{q^m}) = 1 - \frac{1}{q^m}$, the projection of $R$ onto $C_{q^m}$ must have order $q^m$.

Thus $U = U_p U_q R$ with $R$ is a single-element sequence whose element has order $p^k q^m$, and $U_p$, $U_q$ are zero-sum free sequences in $H_p$ and $C_{q^m}$ respectively. This completes the proof on the structure.
\end{proof}

A similar technique works for the second part of Theorem~\ref{thm:result2}.

\begin{proof}[Proof of Theorem~\ref{thm:result2}, Part 2]

We show Theorem~\ref{thm:result2} for $G = C_p \bigoplus C_q \bigoplus C_r$ where $p,q,r$ are distinct primes. We proceed as follows.

\textit{Step 1.} We first wish to show that $\mathsf{k}(G) = \mathsf{k}^{\star}(G)$ for $G = C_p \bigoplus C_q \bigoplus C_r$. Choose the zero-sum free sequence $S$ in $G$ with maximal cross number. It suffices to show $\mathsf{k}(S) \leq \mathsf{k}^{\star}(G)$. Denote by $a(n)$ the number of elements of order $n$ in the sequence $S$. Write $S = S_p S_q S_r S_{pq} S_{pr} S_{qr} S_{pqr}$ where $S_i$ is the subsequence of $S$ consisting of elements of order $i$.

Project the subsequence $S_{q} S_{pq}$ onto the $C_{q}$-coordinate (denote this projection $\tau_1 : G \mapsto C_q$). The resulting sequence over $C_q$ has cross number $\frac{a(q)+a(pq)}{q}$. But every subsequence with cross number at least $1$ in $C_q$ will have a zero-sum subsequence with cross number at most $1$. Therefore, we can find $\lfloor \frac{a(q)+a(pq)}{q} \rfloor$ nonoverlapping zero-sum subsequences in the projection of $S_q S_{pq}$.

In the original sequence $S$, replace the preimage of these zero-sum subsequences with their respective sums: denote this replacement $S_q S_{pq} \mapsto Q_1 R_1$ where $Q_1$ is the ``replaced" sums and $R_1$ is the ``leftover" elements that have not been replaced. Call the resulting sequence $S_1 = S_p Q_1 R_1 S_r S_{pr} S_{qr} S_{pqr}$. Elements of $Q_1$ have order $p$ and $Q_1$ has length $\lfloor \frac{a(q)+a(pq)}{q} \rfloor$.

Project the sequence $R_1 S_{qr} S_{pqr}$ onto the $C_{q}$ coordinate ($\tau_1$). The resulting sequence has cross number at least $\{\frac{a(q)+a(pq)}{q}\} + \frac{a(qr) + a(pqr)}{q}$, the first term from $R_1$ since it was the ``remainder" from the projection above. Using analogous reasoning as above, we can find $\lfloor \{\frac{a(q)+a(pq)}{q}\} + \frac{a(qr) + a(pqr)}{q} \rfloor$ nonoverlapping zero-sum sequences in this projection.

In the sequence $S_1$, replace the preimage of these zero-sum subsequences with its respective sum: denote this replacement $R_1 S_{qr} S_{pqr} \mapsto Q_2 R_2$ where $Q_2$ is the ``replaced" sums and $R_2$ is the ``leftover" elements that have not been replaced. Call the resulting sequence $S_2 = S_p Q_1 Q_2 R_2 S_r S_{pr}$. Elements of $Q_2$ have order dividing $pr$ and $Q_2$ has length $\lfloor \{\frac{a(q)+a(pq)}{q}\} + \frac{a(qr) + a(pqr)}{q} \rfloor$.

Now, project the sequence $Q_2 S_r S_{pr}$ onto the $C_{r}$ coordinate (denote this projection $\tau_2 : G \mapsto C_r$). The resulting sequence over $C_r$ has cross number 
$$\frac{\lfloor \{\frac{a(q)+a(pq)}{q}\} + \frac{a(qr) + a(pqr)}{q} \rfloor}{r} + \frac{a(r)+a(pr)}{r}.$$
By an analogous logic, we can find
$$\left\lfloor \frac{\lfloor \{\frac{a(q)+a(pq)}{q}\} + \frac{a(qr) + a(pqr)}{q} \rfloor}{r} + \frac{a(r)+a(pr)}{r}\right\rfloor$$
\noindent nonoverlapping zero-sum sequences in this projection.

In the sequence $S_2$, replace the preimage of these zero-sum subsequences with its respective sum: denote this replacement $Q_2 S_r S_{pr} \mapsto Q_3 R_3$ where $Q_3$ is the replaced sums and $R_3$ the leftover elements. Call the resulting sequence $S_3 = S_p Q_1 Q_3 R_2 R_3$.

Now, the transformation $S \mapsto S_3$ consists only of replacing some elements by their sum, so $S_3$ still is zero-sum free. But the subsequence $S_p Q_1 Q_3$ is a zero-sum sequence with all its elements in $C_p$ (the other coordinates are zero).

Thus we must have
\begin{align*}
1 - \frac{1}{p} & \geq \mathsf{k}(S_p Q_1 Q_3) = \mathsf{k}(S_p) + \mathsf{k}(Q_1) + \mathsf{k}(Q_3) \\
& \geq \frac{a(p)}{p} + \frac{1}{p} \left\lfloor \frac{a(q)+a(pq)}{q} \right\rfloor + \frac{1}{p} \left\lfloor \frac{\lfloor \{\frac{a(q)+a(pq)}{q}\} + \frac{a(qr) + a(pqr)}{q} \rfloor}{r} + \frac{a(r)+a(pr)}{r}\right\rfloor \\
& \geq \frac{a(p)}{p} + \frac{1}{p} \left\lfloor \frac{a(q)+a(pq)}{q} \right\rfloor \\
& + \frac{1}{p} \left( \frac{\lfloor \{\frac{a(q)+a(pq)}{q}\} + \frac{a(qr) + a(pqr)}{q} \rfloor}{r} + \frac{a(r)+a(pr)}{r} + \frac{1}{r} - 1 \right) \\
& \geq \frac{a(p)}{p} + \frac{1}{p} \left\lfloor \frac{a(q)+a(pq)}{q} \right\rfloor \\
& + \frac{1}{p} \left( \frac{ \{\frac{a(q)+a(pq)}{q}\} + \frac{a(qr) + a(pqr)}{q} + \frac{1}{q} - 1}{r} + \frac{a(r)+a(pr)}{r} + \frac{1}{r} - 1 \right) \\
& = \frac{a(p)}{p} + \frac{1}{p} \left\lfloor \frac{a(q)+a(pq)}{q} \right\rfloor + \frac{1}{pr}  \left\lbrace\frac{a(q)+a(pq)}{q}\right\rbrace \\
& + \frac{a(qr)}{pqr} + \frac{a(pqr)}{pqr} + \frac{a(r)}{pr} + \frac{a(pr)}{pr} + \frac{1}{pqr} - \frac{1}{p} \\
& = \frac{a(p)}{p} + \frac{r-1}{pr} \left\lfloor \frac{a(q)+a(pq)}{q} \right\rfloor + \frac{1}{pr} \frac{a(q)+a(pq)}{q} \\
& + \frac{a(qr)}{pqr} + \frac{a(pqr)}{pqr} + \frac{a(r)}{pr} + \frac{a(pr)}{pr} + \frac{1}{pqr} - \frac{1}{p} \\
& \geq \frac{a(p)}{p} + \frac{r-1}{pr}\left(\frac{a(q)+a(pq)}{q} + \frac{1}{q} - 1\right) + \frac{1}{pr} \frac{a(q)+a(pq)}{q} \\
& + \frac{a(qr)}{pqr} + \frac{a(pqr)}{pqr} + \frac{a(r)}{pr} + \frac{a(pr)}{pr} + \frac{1}{pqr} - \frac{1}{p} \\
& = \frac{a(p)}{p} + \frac{a(pq)}{pq} + \frac{a(pr)}{pr} + \frac{a(pqr)}{pqr} + \frac{a(q)}{pq} + \frac{a(r)}{pr} + \frac{a(qr)}{pqr} + \frac{1}{pq} + \frac{1}{pr} - \frac{2}{p}.
\end{align*}

\noindent Now, if we add $$\frac{p-1}{p} \left( 1 - \frac{1}{q} + 1 - \frac{1}{r} \right) \geq \frac{p-1}{p} \left( \frac{a(q)}{q} + \frac{a(r)}{r} + \frac{a(qr)}{qr}\right)$$ and reorganize terms, we have

$$ \mathsf{k}^{\star}(G) =  3 - \frac{1}{p} - \frac{1}{q} - \frac{1}{r} \geq \frac{a(p)}{p} + \frac{a(pq)}{pq} + \frac{a(pr)}{pr} + \frac{a(pqr)}{pqr} + \frac{a(q)}{q} + \frac{a(r)}{r} + \frac{a(qr)}{qr} = \mathsf{k}(S)$$
\noindent which proves Step 1 for $G = C_p \bigoplus C_q \bigoplus C_r$.


\textit{Step 2.} We now show that each zero-sum free sequence $S$ with $\ms{k}(S) = \ms{k}^\star (G)$ has the form $S = S_p S_q S_r$ where $S_k$ is a zero-sum free sequence over the $k$-primary component of $G$.

Note that in Step 1, we used the fact that $\ms{k}(C_{pq}) = \ms{k}^\star(C_{pq})$ holds with maximal sequences being $S = S_p S_q$. Thus we must have that any zero-sum free sequence $S$ must contain $S_p$ and $S_q$. But the order of $p,q,r$ was arbitrary in Step 1, so we may swap the orders to have that $S$ must also contain $S_r$. Since $\ms{k}(S) = \ms{k}(G) = \ms{k}(S_p) + \ms{k}(S_q) + \ms{k}(S_r)$ and the three sequences are disjoint, they must comprise $S$.

\textit{Step 3.} Now we proceed to prove that $\ms{K}(G) = \ms{K}^\star(G)$. Choose a minimal zero-sum sequence $U$ with $\ms{k}(U) = \ms{K}(G)$.

Do the same transformation as we did in Step 1: replace $U$ with
$$U = U_p U_q U_r U_{pq} U_{pr} U_{qr} U_{pqr} \mapsto U_p Q_1 Q_3 R_2 R_3 = U'.$$
Now if $|R_2 R_3| > 0$, then $U_p Q_1 Q_3$ is a zero-sum free sequence and thus $\ms{k}(U_p Q_1 Q_3) \leq \ms{k}(C_p)$, from which we can proceed as we did in Step 1 to obtain $\ms{k}(S) \leq \ms{k}(G) < \ms{K}(G)$.

Now assume $|R_2 R_3| = 0$. Then $|R_2| = |R_3| = 0$, which implies

$\left\lfloor \frac{\lfloor \{\frac{a(q)+a(pq)}{q}\} + \frac{a(qr) + a(pqr)}{q} \rfloor}{r} + \frac{a(r)+a(pr)}{r}\right\rfloor = \frac{\lfloor \{\frac{a(q)+a(pq)}{q}\} + \frac{a(qr) + a(pqr)}{q} \rfloor}{r} + \frac{a(r)+a(pr)}{r}$ \\
and $\lfloor \{\frac{a(q)+a(pq)}{q}\} + \frac{a(qr) + a(pqr)}{q} \rfloor = \{\frac{a(q)+a(pq)}{q}\} + \frac{a(qr) + a(pqr)}{q}$.

Then we have
\begin{align*}
1 & = \ms{K}(C_p) \geq \ms{k}(U_p Q_1 Q_3) =  \ms{k}(S_p) + \ms{k}(Q_1) + \ms{k}(Q_3) \\
& = \frac{a(p)}{p} + \frac{1}{p} \left\lfloor \frac{a(q)+a(pq)}{q} \right\rfloor + \frac{1}{p} \left( \frac{\{\frac{a(q)+a(pq)}{q}\} + \frac{a(qr) + a(pqr)}{q}}{r} + \frac{a(r)+a(pr)}{r} \right) \\
& \geq \frac{a(p)}{p} + \frac{a(pq)}{pq} + \frac{a(pr)}{pr} + \frac{a(pqr)}{pqr} + \frac{a(q)}{pq} + \frac{a(r)}{pr} + \frac{a(qr)}{pqr} - \frac{1}{p}\left( 1 - \frac{1}{q} \right)\left( 1 - \frac{1}{r} \right)
\end{align*}

Now since $U_q U_r U_{qr}$ can't constitute $U$ ($U$ has maximal cross number among irreducible zero-sum sequences by construction), we have
$$\ms{k}(U_q U_r U_{qr}) \leq \ms{k}(C_{qr}).$$ Thus If we add
$$\frac{p-1}{p} \left( 1 - \frac{1}{q} + 1 - \frac{1}{r} \right) \geq \frac{p-1}{p} \left( \frac{a(q)}{q} + \frac{a(r)}{r} + \frac{a(qr)}{qr}\right)$$ to both sides and rearrange terms, we have
$$\ms{K}^\star (G) = 3 - \frac{1}{p} - \frac{1}{q} - \frac{1}{r} + \frac{1}{pqr} \geq \ms{k}(U),$$
which is exactly the desired inequality.

\textit{Step 4.} Now we show the structural result, which immediately follow from Step 3: the equality conditions yield $U = g S_p S_q S_r$ for appropriate $S_p, S_q, S_r$ which are zero-sum free sequences of $C_p, C_q, C_r$ respectively.
\end{proof}

\section{Concluding Remarks}

While a verification of the conjecture for more families of groups seems within reach, a full proof of Conjecture~\ref{conj:main} seems far away. The study of minimal zero-sum sequences over general finite abelian groups is limited by two obstacles, namely the number of prime divisors of $\exp(G)$ and the rank of $G$. Previous works such as Theorem~\ref{thm:prevresult} and Theorem~\ref{thm:hegroup} prove Conjecture~\ref{conj:main} for groups with $\exp(G)$ having only one prime divisor (\cite{G2}), two divisors with some restrictions (\cite{GS2}), or small rank (\cite{GS2}, \cite{H}). Our work verifies Conjecture~\ref{conj:main} for more families of groups with $\exp(G)$ having two prime divisors (Theorem~\ref{thm:result2}) and groups of rank 2 (Theorem~\ref{thm:result1}). The study is especially obstacled when both the rank and the number of prime divisors of $\exp(G)$ are large. In this section, we note some observations that may help resolve this obstacle.

Note that the proof of Theorem~\ref{thm:result2} immediately implies the following:

\begin{cor}
\label{cor:nocross}
For $G = C_p \bigoplus C_q \bigoplus C_r$ and $G = H_p \bigoplus C_{q^m}$, the only minimal zero-sum sequences $S$ that satisfy $\mathsf{k}(S) = \mathsf{K}(G)$ are sequences $S$ with $S = g S'$, where $g$ is an element of order $\exp(G)$ and $S'$ is a zero-sum sequence without terms of order divisible by two or more primes.
\end{cor}

This motivates us to conjecture the following strengthening of Conjecture~\ref{conj:main}, with respect to the \textit{structure} of zero-sum free sequences and minimal zero-sum sequences:

\begin{conj}
\label{conj:nocross}
Let $G$ be a finite abelian group with primes dividing $\exp(G)$ be $p_1, p_2, \cdots, p_s$. Then we have
\be
\ii Each minimal zero-sum sequence $U$ over $G$ with $\ms{k}(U) = \ms{K}(G)$ has the form $U = g \prod U_i$ where $U_i$ is a zero-sum free sequence over the $p_i$-primary component of $G$ for each $i \in [1,s]$.
\ii In particular, $\ms{K}(G) = \frac{1}{\exp(G)} + \ms{k}^\star (G)$ and each zero-sum free seuqnece $S$ with $\ms{k}(S) = \ms{k}^\star (G)$ has the form $S = \prod S_i$ where $S_i$ is a zero-sum free sequence over the $p_i$-primary component of $G$ for each $i \in [1,s]$.
\ee
\end{conj}

To see that Conjecture~\ref{conj:nocross} implies Conjecture~\ref{conj:main}, note that Conjecture~\ref{conj:nocross} immediately implies $\mathsf{k}(\bigoplus_{i=1}^{r} H_{p_i}) = \sum_{i=1}^{r} \mathsf{k}(H_{p_i})$ for any $p_i$-groups $H_{p_i}$ and apply Theorem~\ref{thm:prevresult} for $p$-groups.

Conjecture~\ref{conj:nocross} motivates us to weight sequences in different manners, because if the conjecture is true, the coefficients on elements of order divisible by two or more primes do not matter for maximal zero-sum free sequences. Consider any arbitrary weighting function $f$ on the positive integers. Extending He's notation (\cite{H}), if $S$ is a sequence over a finite abelian group $G$ define

$$\mathsf{k}(S,f) = \sum_{g \in G} v_g (S) f(\ord(g)),$$
\noindent and define $\mathsf{k}(G,f)$ naturally. We know that $f = f_1(n) = \frac{1}{n}$ correspond to the cross number we know, but if we define $f = f_g(n) = \frac{g(n)}{n}$ where $g$ is a function from the natural numbers that satisfy $g(p^k) = 1$ for all prime powers $p^k$, then Conjecture~\ref{conj:nocross} implies that $\mathsf{k}(S,f_0) = \mathsf{k}(S, f_g)$ for sequences $S$ such that $\mathsf{k}(S) = \mathsf{k}(G)$. An appropriate choice of weighting functions $f,g$ might shed light on both Conjecture~\ref{conj:main} and Conjecture~\ref{conj:nocross}.

\section{Acknowledgements}
This research was supervised by Joe Gallian at the University of Minnesota Duluth REU, supported by the National Science Foundation (grant number DMS-1358659) and the National Security Agency (grant number H98230-13-1-0273). I would like to thank the Department of Mathematics at Princeton University, and the Samsung Scholarship Foundation for their funding for the independent work. I would like to thank Joe Gallian for his advices and support. I would also like to thank the program advisors Adam Hesterberg, Noah Arbesfeld, Daniel Kriz, and fellow participant Xiaoyu He for their valuable suggestions to this paper. I also would like to thank the anonymous referee for extremely helpful advices and comments on the paper, and Benjamin Girard for his e-mail communication and advices.

\end{document}